\documentclass[12pt]{amsart}
\usepackage{epsfig,color}

\makeatother

\theoremstyle{definition}

\def\fnum{equation}
\newtheorem{Thm}[\fnum]{Theorem}
\newtheorem{Cor}[\fnum]{Corollary}

\newtheorem{Lem}[\fnum]{Lemma}

\newtheorem{Pro}[\fnum]{Proposition}

\numberwithin{equation}{section}
\newcommand{\nn}{{\bf{n}}}

\def\RR{{\bold R}}

\def\SS{{\bold S}}

\newcommand{\dv}{{\text {div}}}
\newcommand{\e}{{\text {e}}}

\newcommand{\Area}{{\text {Area}}}

\newcommand{\cS}{{\mathcal{S}}}

\newcommand{\eqr}[1]{(\ref{#1})}

\title[Compactness of self-shrinkers]{Smooth compactness of   self-shrinkers}

\author{Tobias H. Colding}%
\address{MIT, Dept. of Math.\\
77 Massachusetts Avenue, Cambridge, MA 02139-4307.}
\author{William P. Minicozzi II}%
\address{Johns Hopkins University\\
Dept. of Math.\\
3400 N. Charles St.\\
Baltimore, MD 21218}

\thanks{The   authors
were partially supported by NSF Grants DMS  0606629 and DMS
0405695}

\email{colding@math.mit.edu  and minicozz@math.jhu.edu}

\begin{document}

\maketitle

\begin{abstract}
We prove a smooth compactness theorem for the space of embedded self-shrinkers in $\RR^3$.
Since self-shrinkers model singularities in mean
curvature flow, this theorem can be thought of as a compactness result for the
space of all singularities and it plays an important role in studying generic mean curvature flow.
\end{abstract}

\section{Introduction}

A  surface $\Sigma \subset \RR^3$ is  said to be a {\emph{self-shrinker}} if it satisfies
\begin{equation}	\label{e:ss}
	H = \frac{ \langle x , \nn \rangle }{2} \, ,
\end{equation}
where $H = \dv \,  \nn$ is the mean curvature, $x$ is the position vector, and $\nn$ is the unit normal.   This is easily seen to be equivalent to that $\Sigma$ is the $t=-1$ time-slice{\footnote{In \cite{H3},  self-shrinkers are  time $t= - \frac{1}{2}$ slices of   self-shrinking  MCFs; these satisfy $H = \langle x , \nn \rangle$.}}  of a mean curvature flow (``MCF'') moving by rescalings, i.e., where the time $t$ slice is given by $\sqrt{-t} \, \Sigma$.

Self-shrinkers play an important role in the study of mean curvature flow.  Not only are they the simplest examples (those where later time slices are rescalings of earlier),  but  they also  describe all possible blow ups at a given singularity of a mean curvature flow.    The  idea is that we can rescale a MCF in space and time to obtain a new MCF, thereby expanding a  neighborhood of the point that we want to focus on.   Huisken's monotonicity, \cite{H3},    and Ilmanen's  compactness theorem, \cite{I1},  give  a subsequence converging to a limiting solution of the MCF; cf. \cite{I1}, \cite{W1}.  This limit, which is called a {\emph{tangent flow}},
 achieves equality in Huisken's monotonicity and, thus,    its time $t$ slice  is $\sqrt{-t} \, \Sigma$ where $\Sigma$ is a  self-shrinker.

The main result of this paper is the following smooth compactness theorem
for self-shrinkers in $\RR^3$ that is used in \cite{CM1}.

\begin{Thm}	\label{c:cpt}
 Given an integer $g \geq 0$ and a  constant $V>0$,
the space  of smooth complete embedded self-shrinkers $\Sigma \subset \RR^3$ with
\begin{itemize}
\item genus at most $g$,
\item $\partial \Sigma = \emptyset$,
\item $\Area \, \left( B_{R}(x_0) \cap \Sigma \right) \leq V \, R^2$  for all $x_0 \in \RR^3$ and all $R > 0$
 \end{itemize}
 is compact.

 Namely, any sequence of these has a subsequence that converges in the
 topology of  $C^m$ convergence on compact subsets for any $m \geq 2$.
\end{Thm}

The surfaces in this theorem are assumed to be homeomorphic to closed surfaces with finitely many disjoint disks removed.
The genus of the surface is defined to be the genus of the corresponding closed surface.  For example,  an annulus is a sphere with two disks removed and, thus, has genus zero.
 Below, we will use that the genus is monotone in the sense that if $\Sigma_1 \subset \Sigma_2$, then the genus of $\Sigma_1$ is at most that of $\Sigma_2$.

  As mentioned, the main motivation for this result is that self-shrinkers model singularities in mean curvature flow.  Thus, the above theorem can be thought of as a compactness result for the space of all singularities.   In practice,  scale-invariant local area bound, smoothness, and the genus bound  will automatically come from corresponding bounds on the initial surface in a MCF.   Namely:
 \begin{itemize}
 \item Area bounds are a direct consequence of Huisken's monotonicity formula, \cite{H3}.{\footnote{See, for instance, corollary $2.13$ in \cite{CM1}.}}
 \item Ilmanen proved that in $\RR^3$  tangent flows at the first singular time must be smooth and have genus at most that of the initial surface; see theorem $2$ of \cite{I1} and page $21$ of \cite{I1}, respectively.
 \end{itemize}
 Conjecturally, the smoothness and genus bound hold at all singular times:
 \begin{itemize}
 \item Ilmanen conjectured that tangent flows are smooth and have multiplicity  one at all singularities.
If this conjecture holds, then it would follow from Brakke's regularity theorem that near a singularity the flow can be written as a graph of a function with small gradient over the tangent flow.  Combining this with the above mentioned monotonicity of the genus of subsets and a result of White, \cite{W3}, asserting that the genus of the evolving surfaces   are always bounded by that of the initial surface, we get conjecturally that the genus of the tangent flow is at most that of the initial surface.
\end{itemize}

Our compactness theorem will play an important role in understanding generic mean curvature flow   in \cite{CM1}.  Namely, in \cite{CM1}, we will see that it follows immediately from compactness together with the classification of (entropy) stable self-similar shrinkers proven in \cite{CM1} that given an integer $m$ and $\delta>0$, there exists an $\epsilon=\epsilon (m,\delta,V,g)>0$ such that:
 \begin{itemize}
 \item For any unstable self-similar shrinker in $\RR^3$ satisfying the assumptions of Theorem \ref{c:cpt}, there is a surface $\delta$-close to it in the $C^m$ topology and with entropy less than that of the self-similar shrinker $-\epsilon$.
 \end{itemize}
 This is, in particular, a key to showing that mean curvature flow that disappears  in a compact point does so generically in a round point; see \cite{CM1} for details and further applications.

 The simplest examples of self-shrinkers in $\RR^3$ are the  plane $\RR^2$, the sphere of radius
$2$, and  the cylinder   $\SS^1 \times \RR^1$ (where the $\SS^1$ has radius $\sqrt{2}$).  Combining \cite{H3}, \cite{H4}, and theorem 0.18 in \cite{CM1} it follows that these are the  only smooth embedded self-shrinkers with $H \geq 0$ and polynomial volume growth.{\footnote{Huisken, \cite{H3}, \cite{H4}, showed that these are the  only smooth embedded self-shrinkers with $H \geq 0$, $|A|$ bounded, and polynomial volume growth.  In \cite{CM1}, we prove that this is the case even without assuming a bound on $|A|$.}}  
It follows from this that spheres and cylinders are isolated (among all self-shrinkers) in the $C^2$-topology. On the other hand, by Brakke's theorem, \cite{Br},  any self-shrinker with entropy sufficiently close to one (which is the entropy of the plane) must be flat, so  planes are also isolated and we see that
all three of the simplest self-shrinkers are isolated.     Moreover, one of the key results of \cite{CM1} (see theorem 0.7 there) was to show that these are the only (entropy) stable self-shrinkers.    
   In sum, self-shrinkers either with $H\geq 0$ or that are stable are one of the three simplest types and all of those are isolated among all self-shrinkers.    \footnote{Both the classification of stable self shrinkers from \cite{CM1} and that those are isolated are implicitly used in the application in \cite{CM1}, mentioned above, of our compactness theorem to prove that the $\epsilon>0$ above can be chosen independently of the self-shrinker and not just independently for all self-shrinkers a definitely distance away from one of the stable ones. }     
  However,   there are expected to be many  examples of self-shrinkers in $\RR^3$ where $H$ changes sign or that are unstable.   In particular, Angenent, \cite{A},  constructed a self-shrinking torus
 of revolution and there is numerical evidence for a number of other examples; cf. Chopp, \cite{Ch}, Angenent-Chopp-Ilmanen, \cite{AChI}, Ilmanen, \cite{I2}, and Nguyen, \cite{N1}, \cite{N2}.  These examples suggest that compactness fails to hold without a genus bound.

\vskip2mm
There are three key ingredients  in the proof of the compactness theorem.  The first is a singular compactness theorem that gives a subsequence that converges to a smooth limit away from a locally finite set of points.
Second, we show that if the convergence is not smooth, then the limiting self-shrinker is $L$-stable, where $L$-stable means that for any compactly supported function $u$ we have
\begin{equation}	\label{e:Lstable}
	\int_{\Sigma} \left( - u \, L \, u \right) \, \e^{ - \frac{|x|^2}{4} } \geq 0 \, .
\end{equation}
Here $L$ is the second order operator from \cite{CM1} that is given by
\begin{equation}
	L \, u = \Delta \, u + |A|^2 \, u  - \frac{1}{2} \, \langle x , \nabla u \rangle + \frac{1}{2} \, u \, .
\end{equation}
The last ingredient is the following result from \cite{CM1}:

\begin{Thm}	\label{t:spectral0}
\cite{CM1}
There are no $L$-stable smooth complete self-shrinkers without boundary and with polynomial volume growth in $\RR^{n+1}$.
\end{Thm}

To keep this paper self-contained, we will   prove   Theorem \ref{t:spectral0} in an appendix.

\vskip2mm
 Finally, we note that
the results of \cite{CM4}--\cite{CM8} suggest that there is a compactness theorem for embedded self-shrinkers even without an area bound.
 However,  as mentioned above, then it follows from Huisken's monotonicity formula that self-shrinkers  arising as tangent flows at singularities of a MCF starting at a smooth closed surface automatically satisfy an area bound for some constant depending only on the initial surface.

\subsection{Conventions and notation}
A one-parameter family $M_t$ of hypersurfaces in $\RR^{n+1}$ flows by mean curvature if
\begin{equation}
	\left( \partial_t X \right)^{\perp} =  - H \, \nn \, ,
\end{equation}
where
  $\nn$ is the outward unit normal and the mean curvature $H$ is given by
$H = \dv \,  \nn$.
With this convention, $H$ is $n/R$ on the $n$-sphere of radius $R$ in $\RR^{n+1}$ and $H$ is $k/R$ on the ``cylinder'' $\SS^k \times \RR^{n-k} \subset \RR^{n+1}$ of radius $R$.
 If $e_i$ is an orthonormal frame for $\Sigma$, the coefficients of the second fundamental form are defined to be
$a_{ij} = \langle \nabla_{e_i} e_j , \nn \rangle$.  In particular, we have
\begin{equation}	\label{e:aij}
	\nabla_{e_i} \nn  = - a_{ij} e_j \, .
\end{equation}
Since $\langle \nabla_{\nn} \nn , \nn \rangle =0$, the mean curvature is   $H = \langle \nabla_{e_i} \nn , e_i \rangle = - a_{ii}$ where by convention we are summing over repeated indices.

\section{The self-shrinker equation}

The starting point for understanding self-shrinkers is to realize that
there are several other ways to characterize self-shrinkers that are equivalent to the equation \eqr{e:ss}:
\begin{enumerate}
	\item The one-parameter family of hypersurfaces $\sqrt{-t} \, \Sigma \subset \RR^{n+1}$ satisfies MCF.
	\item $\Sigma$ is a minimal hypersurfaces in $\RR^{n+1}$, not with the Euclidean metric $\delta_{ij}$, but with the conformally changed metric
$
	g_{ij}  = \e^{\frac{-|x|^2}{2n}}  \, \delta_{ij} \, .
$
\item $\Sigma$ is a critical point for the functional $F$ defined on a hypersurface $\Sigma \subset \RR^{n+1}$ by
 \begin{equation}	\label{e:F}
 	F  (\Sigma) = (4\pi)^{-n/2} \, \int_{\Sigma} \, \e^{\frac{-|x|^2}{4}} \, d\mu \, .
\end{equation}
\end{enumerate}
The characterization (2) is particularly useful since it will allow us to use local estimates and compactness theorems for minimal surfaces to get corresponding results for self-shrinkers.

\subsection{The equivalence of (1), (2), and (3)}
The fact that (1), (2), and (3) are equivalent to satisfying the self-shrinker equation \eqr{e:ss} is well-known, but we will include a short proof of this in the next two lemmas.

 \begin{Lem}	\label{l:ss}
If a hypersurface $\Sigma$ satisfies  \eqr{e:ss}, then $M_t=\sqrt{-t}\,\Sigma$ satisfies MCF and
\begin{equation} \label{e:selfs}
H_{M_t}=-\frac{\langle  x,\nn_{M_t}\rangle}{2t}\, .
\end{equation}
Conversely, if  $M_t$ is an MCF, then  $M_t = \sqrt{-t} \, M_{-1}$
 if and only if
  $M_{t}$ satisfies \eqr{e:selfs}.
\end{Lem}

\begin{proof}
If $\Sigma$ is a hypersurface that satisfies \eqr{e:ss}, then we set $M_t=\sqrt{-t}\,\Sigma$ and $x(p,t)=\sqrt{-t}\,p$ for $p\in \Sigma$.  It follows that $\nn_{M_t}(x(p,t))=\nn_{\Sigma}(p)$, $H_{M_t}(x(p,t))=\frac{H_{\Sigma}(p)}{\sqrt{-t}}$, and $\partial_t x=-\frac{p}{2\sqrt{-t}}$.  Thus, $(\partial_t x)^{\perp}=-\frac{\langle p,\nn\rangle}{2\sqrt{-t}}=-H_{M_t}(x(p,t))$.  This proves that $M_t$ is an MCF and shows \eqr{e:selfs}.

On the other hand, suppose that $M_t$ is an MCF.
A computation shows that
\begin{equation}
(-t)^{\frac{3}{2}} \, \partial_t\left( \frac{x}{\sqrt{-t}}\right)=-t\, \partial_t x +\frac{x}{2} \, .
\end{equation}
If $\frac{M_t}{\sqrt {-t}}=M_{-1}$, then
\begin{equation}
0=(-t)^{\frac{3}{2}} \, \langle \partial_t\left( \frac{x}{\sqrt{-t}}\right),\nn_{M_{-1}}\rangle=-t\,\langle  \partial_t x,\nn_{M_{-1}}\rangle +\frac{1}{2}\langle x,\nn_{M_{-1}}\rangle\, .
\end{equation}
Hence, since $M_t$ is an MCF,  it follows that
\begin{equation}
H_{M_{-1}}=-\langle \partial_t x,\nn_{M_{-1}}\rangle = \frac{\langle x , \nn_{M_{-1}} \rangle }{2} \, .
\end{equation}
The equation for $H_{M_t}$ for general $t$ follows by scaling.

Finally, if an MCF $M_t$  satisfies \eqr{e:selfs}, then, by the first part of the
lemma,  $N_t=\sqrt{-t}M_{-1}$ is  an MCF with the same
initial condition as $M_t$; thus $M_t=N_t$ for $t\geq -1$.
\end{proof}

The next lemma, which is due to Huisken, \cite{H3} (cf. Ilmanen, page $6$ of \cite{I2}, \cite{A}; see also  \cite{CM1}), computes the first variation of the $F$ functional; since it is so short, we include the proof here.  The equivalence  of both (2) and (3) with \eqr{e:ss} follows  from this lemma.

 \begin{Lem}	\label{l:varl0}
If $x' = f \nn$ is a compactly supported normal variation of a hypersurface $\Sigma \subset \RR^{n+1}$
and $s$ is the variation parameter, then $\frac{\partial }{\partial s} \,    F  (\Sigma_s)  $ is
\begin{equation}	\label{e:Fprime0}
	  (4\pi)^{-\frac{n}{2}} \,  \int_{\Sigma}   f \, \left( H - \frac{ \langle x ,  \nn \rangle}{2}
		\right)  \, d\mu \, .
\end{equation}
\end{Lem}

\begin{proof}
The first variation formula (for volume) gives
 \begin{equation}	\label{e:fhdm}
 	(d\mu)' = f \, H \, d \mu \, .
 \end{equation}
 The $s$ derivative of   $\log \, \left[ (4\pi )^{-\frac{n}{2}} \, \e^{- \frac{|x|^2}{4}} \right]$ is given by
 $
 	  - \frac{ f}{2} \, \langle x , \nn \rangle   \, .
$
 Combining this with \eqr{e:fhdm} gives \eqr{e:Fprime0}.
\end{proof}

 \subsection{Self-shrinkers as minimal surfaces}

 We saw that  self-shrinkers in $\RR^{n+1}$ are minimal hypersurfaces for the conformally changed metric
 \begin{equation}
 	g_{ij} = \e^{- \frac{ |x|^2}{2n} } \, \delta_{ij}  \, .
\end{equation}
 We will use this in the next section to get local estimates and singular compactness, but first investigate these metrics a bit.  In particular,
we will see that these metrics cannot be made complete and, thus, the compactness of the space of self-shrinkers does not follow from compactness results for minimal surfaces such as the Choi-Schoen, \cite{CS}, compactness for positive Ricci curvature; cf. \cite{CM2}.    In fact, it turns out the Ricci curvature of these metrics does not  have a sign and goes to {\underline{negative infinity}} at infinity.

  We begin with the obvious observation that the distance to infinity is finite since
$\int_0^{\infty} \, \e^{- \frac{ t^2}{4n} } \, dt < \infty$.  Furthermore, for $n \geq 2$, the scalar curvature $\tilde{R}$ of the metric
 $u^{\frac{4}{n-1}}\, \delta_{ij}$ is given by{\footnote{See  page $184$ in \cite{SY}; the formula there is for an $n$-dimensional manifold, so we have shifted $n$ by one.}}
\begin{equation}
	\tilde{R} =   \frac{-4n}{n-1} \, u^{ \frac{ -(n+3)}{n-1}} \, \Delta \, u \, .
\end{equation}
Thus, for our conformal metrics, we have  $u = \e^{ \frac{ (1-n) \, |x|^2 }{8\, n}} $.  Using that $\Delta \e^f = \e^f ( \Delta f + |\nabla f|^2)$,
$\Delta |x|^2 = 2(n+1)$ on $\RR^{n+1}$, and $|\nabla |x|^2|^2 = 4 \, |x|^2$, we get that
\begin{equation}
	\Delta \, u = u \, \left( \frac{   (n-1)^2 }{16\, n^2} \, |x|^2 -   \frac{ n^2 -1}{4\, n} \right) \, .
\end{equation}
 It follows that the scalar curvature is
 \begin{equation}	\label{e:scalar}
	\tilde{R} =    u^{ \frac{ -4}{n-1}} \,  \left( n+1  - \frac{n-1}{4\, n} \, |x|^2 \right) =
		 \e^{\frac{|x|^2}{2n}} \,   \left(  n+1  - \frac{n-1}{4\, n} \, |x|^2 \right)  \, .
\end{equation}
There are a few interesting consequences of this formula.  First, the scalar curvature does not have a sign; it is positive when $|x|$ is small and then becomes negative near infinity.  Second, as $|x| \to \infty$, the scalar curvature goes to negative infinity.  It follows that the space is not complete; even though infinity is at a finite distance, there is no way to smoothly extend the metric to a neighborhood of infinity.

\section{Compactness away from a locally finite set of singular points}

We specialize now to   self-shrinkers in $\RR^3$.
We will use the following well-known local singular compactness for embedded minimal surfaces  in any Riemannian $3$-manifold.

\begin{Pro}	\label{p:cp1}
Given a point $p$ in a Riemannian $3$-manifold $M$, there exists $R>0$ so:\\
Let $\Sigma_j$ be embedded minimal surfaces in $B_{2R} = B_{2R} (p) \subset M$  with $\partial \Sigma_j\subset \partial B_{2R}$.  If
 each $\Sigma_j$ has area at most $V$ and genus at most $g$ for some fixed $V, g$, then there  is a finite collection of points $x_k$, a smooth embedded minimal surface $\Sigma \subset B_R$ with $\partial \Sigma \subset \partial B_R$ and a subsequence of the $\Sigma_j$'s that converges in $B_R$ (with finite multiplicity) to $\Sigma$ away from the $x_k$'s.
\end{Pro}

 There are a number of ways to prove this proposition.  For instance, one can use the bounds on the area and genus to get uniform total curvature bounds on $B_{3R/2} \cap \Sigma_j$ (this follows from the local Gauss-Bonnet estimate given in theorem $3$ of \cite{I1}) and then argue as in Choi-Schoen, \cite{CS}.  Alternatively, the proposition is an immediate consequence of the much more general compactness results of \cite{CM4}--\cite{CM8} that hold  even without the area bound.

Combining  Proposition \ref{p:cp1}   with a covering argument (and going to a diagonal subsequence) gives a global singular compactness theorem for self-shrinkers:

\begin{Cor}	\label{c:cp1}
Suppose that
$\Sigma_i \subset \RR^3$ is a sequence of smooth embedded complete self-shrinkers with genus at most $g$,  $\partial \Sigma_i = \emptyset$, and
the scale-invariant area bound for all $x_0 \in \RR^3$ and all $R>0$
\begin{equation}
	 \Area \, \left( B_{R}(x_0) \cap \Sigma_i \right) \leq V\, R^2
	  \, .
\end{equation}
Then there is a subsequence (still denoted by $\Sigma_i$), a smooth  embedded complete (non-trivial) self-shinker $\Sigma$ without boundary, and a locally finite collection of points $\cS \subset \Sigma$ so that $\Sigma_i$ converges smoothly (possibly with multiplicity) to $\Sigma$ off of $\cS$.
\end{Cor}

\begin{proof}
The compactness follows by covering $\RR^3$ by a countable collection of small balls on which we can apply
  Proposition \ref{p:cp1} and then passing to a   diagonal subsequence.  To see that the limit must be non-trivial, observe that every self-shrinker must intersect the closed ball bounded by the spherical self-shrinker.  This follows from the maximum principle since the associated MCF's both disappear at the same point in space and time.
   \end{proof}

A set $\cS \subset \RR^3$ is said to be {\emph{locally finite}} if $B_R \cap \cS$ is finite for every $R > 0$.

\section{Showing that the convergence is smooth}

 It remains to show that the convergence is smooth everywhere.  By Allard's theorem, \cite{Al}, this follows from showing that the multiplicity must be one.  We will show that if the multiplicity is greater than one, then the limit $\Sigma$ is   $L$-stable where
 \begin{equation}	\label{e:defL31}
	L = \Delta + |A|^2 - \frac{1}{2} \, \langle x , \nabla (\cdot) \rangle + \frac{1}{2}
\end{equation}
is the linearization of the self-shrinker equation (see \cite{CM1}).

\begin{Pro}	\label{p:jacobi}
If the multiplicity of the convergence of the $\Sigma_i$'s in Corollary \ref{c:cp1} is greater than one, then
$\Sigma$ is $L$-stable.
 \end{Pro}

   The  idea for the proof of Proposition \ref{p:jacobi} comes from a related argument for minimal surfaces in  \cite{CM9}.

\begin{proof}
(of Proposition \ref{p:jacobi}).
Since the limit surface $\Sigma \subset \RR^3$ is complete, properly embedded, and has no boundary,
$\Sigma$ separates $\RR^3$ and has a well-defined unit normal $\nn$.  By assumption, the convergence of the $\Sigma_i$'s to $\Sigma$ is {\underline{not}} smooth and, thus, by Allard's theorem
\cite{Al} must have multiplicity greater than one.

{\bf{Existence of a positive solution $u$ of $L \, u = 0$}}.
Let  $\cS$ be the
  (non-empty) locally finite collection of singular points  for the convergence.
Since the convergence is smooth away from the $y_i$'s,
we can choose $\epsilon_i \to 0$ and domains
$\Omega_i \subset \Sigma$
exhausting $\Sigma \setminus
\cS$ so that each $\Sigma_i$ decomposes locally as
a collection of graphs over $\Omega_i$ and is contained in the
$\epsilon_i$ tubular neighborhood of $\Sigma$.
By embeddedness (and orientability), these sheets are ordered by
height.  Let $w^+_i$ and $w^-_i$ be the functions representing
the top and bottom sheets over $\Omega_i$.
Arguing as in
equation (7) of \cite{Si2}, the difference $w_i = w^+_i - w^-_i$
satisfies $L w_i = 0$ up to higher order correction terms since the operator $L$ given by \eqr{e:defL31} is the linearization of the self-shrinker equation (this is proven in appendix A in \cite{CM1}).

Fix some $y \notin \cS$ and set
$u_i =  w_i \, / \, w_i (y)$.
Since the $u_i$'s are positive (i.e., the sheets are disjoint),
the Harnack inequality implies local $C^{\alpha}$ bounds
(theorem 8.20 of \cite{GiTr}).
Elliptic theory then gives $C^{2 , \alpha}$ estimates
(theorem 6.2 of \cite{GiTr}).
By the Arzela-Ascoli theorem,
a subsequence converges uniformly in $C^2$ on compact subsets of $\Sigma \setminus \cS$ to a non-negative function $u$ on $\Sigma \setminus \cS$ which satisfies
\begin{equation}
	L u = 0 {\text{ and }} u(y) = 1 \, .
\end{equation}
It remains to show that $u$ extends smoothly
across the $y_k$'s to a solution of $L u = 0$.
This follows from standard removable singularity results for elliptic equations once we show that $u$ is bounded up to each $y_k$.
Consider the cylinder $N_k$ (in exponential normal coordinates) over
$B_{\epsilon}(y_k) \subset \Sigma$.  If $\epsilon$ is
sufficiently small, then
a result
of White (see the appendix of \cite{W2})
gives a foliation by minimal (in the conformal metric) graphs $v_t$ of some normal neighborhood
of $\Sigma$ in $N_k$ so that
\begin{align}   \label{e:white}
        v_0 (x) &= 0
                {\text{ for all }} x \in B_{\epsilon}(y_k) \, , \,
{\text{and}}  \\
        v_t (x) &= t
                {\text{ for all }} x \in \partial B_{\epsilon}(y_k) \, .
\notag
\end{align}
Furthermore, the Harnack inequality
implies that
$t / C \leq v_t \leq C \, t$ for some $C > 0$.
In particular, combining \eqr{e:white} with the maximum principle for minimal surfaces (and the
Hausdorff convergence of the $\Sigma_i$'s to $\Sigma$), we see that $u_i$ is
bounded on
$B_{\epsilon}(y_k)$ by a multiple of its supremum on
$B_{\epsilon}(y_k) \setminus
 B_{\epsilon / 2}(y_k)$.  We conclude that $u$ has a removable singularity at each $y_i$ and thus extends to a non-negative solution of $L u = 0$ on all of $\Sigma$; since $u(y) =1$, the Harnack inequality implies that $u$ is everywhere positive.

 {\bf{Using $u$ to prove $L$-stability}}.
 We will now use a variation on an argument of Fischer-Colbrie-Schoen (see, e.g., proposition $1.26$ in \cite{CM2}). Set $w = \log u$, so that
\begin{equation}	\label{e:logu}
	\Delta w = \frac{\Delta u}{u} - |\nabla w |^2 = -|A|^2  + \frac{1}{2} \langle x , \nabla w \rangle - \frac{1}{2} - |\nabla w |^2 \, .
\end{equation}
	Given  $\phi$ with compact support,   applying Stokes' theorem to $ \dv \left( \phi^2 \,  \e^{\frac{-|x|^2}{4}}  \, \nabla w \right) $ gives
\begin{align}	\label{e:stable}
	0 &=  \int \left( 2 \, \phi \langle \nabla \phi , \nabla w \rangle
		+ \left[ - |A|^2  - \frac{1}{2}  - |\nabla w |^2 \right]  \, \phi^2  \right) \, \e^{\frac{-|x|^2}{4}}  \notag \\
		&\leq  \int \left( | \nabla \phi |^2
		 - |A|^2 \, \phi^2  - \frac{1}{2}    \, \phi^2  \right) \, \e^{\frac{-|x|^2}{4}}  = -
		  \int_{\Sigma} (\phi \, L \phi) \, \e^{\frac{-|x|^2}{4}}  \, ,
\end{align}
where the inequality used   $ 2 \, \phi \langle \nabla \phi , \nabla w \rangle  \leq \phi^2 \, |\nabla w|^2 + |\nabla \phi|^2$ and the last equality came from applying Stokes' theorem to $\dv \, \left( \phi \, \nabla \phi \, \e^{\frac{-|x|^2}{4}} \right)$.
\end{proof}

\begin{proof}
(of Theorem \ref{c:cpt}).  We will argue by contradiction.  Suppose therefore that there is a
sequence
 of smooth complete embedded self-shrinkers $\Sigma_i \subset \RR^3$ with genus $g$,  $\partial \Sigma = \emptyset$, and
the scale-invariant area bound for all $x_0 \in \RR^3$ and all $R > 0$
\begin{equation}
	 \Area \, \left( B_{R}(x_0) \cap \Sigma_i \right)  \leq V \, R^2
	 \, ,
\end{equation}
	but so that $\Sigma_i$ does not have any smoothly convergent subsequences.
	By Corollary  \ref{c:cp1}, we can pass to a subsequence so that the $\Sigma_i$'s converge (possibly with multiplicity) to an embedded self-shrinker $\Sigma$ away from a locally finite set $\cS \subset \Sigma$.  By assumption, $\cS$ is non-empty and, by Allard's theorem, the convergence has multiplicity greater than one.  Consequently, Proposition \ref{p:jacobi} implies that $\Sigma$ is $L$-stable.  However, Theorem \ref{t:spectral0} gives that no such $\Sigma$ exists, giving the desired contradiction.
\end{proof}

\appendix

\section{There are no $L$-stable self-shrinkers}

In this appendix, we will include a   proof of Theorem \ref{t:spectral0} from \cite{CM1} for the reader's convenience.   Throughout,    the smooth complete embedded
 hypersurface $\Sigma \subset \RR^{n+1}$ will be  a self-shrinker without boundary and with polynomial volume growth.

We will need  the following calculation from \cite{CM1}:  The   normal part $\langle v , \nn \rangle$ of a constant vector field $v $  is an eigenfunction of $L$ with
\begin{equation}	\label{e:spec1}
	 L \langle v , \nn \rangle =  \frac{1}{2} \,  \langle v , \nn \rangle \, .
\end{equation}

\begin{proof}
(of Theorem \ref{t:spectral0}).
We will construct a  compactly supported function $u$ that does not satisfy   \eqr{e:Lstable}.  Fix a point $p$ in $\Sigma$ and define a function   $v$ on $\Sigma$ by
\begin{equation}
	v(x)  = \langle \nn (p) , \nn (x) \rangle \, .
\end{equation}
It follows that $v (p) = 1$, $|v| \leq 1$, and, by \eqr{e:spec1},
that
	$L \, v = \frac{1}{2} \, v $.
	Therefore, given any smooth function $\eta$, we have
\begin{align}
	L \, (\eta \, v) &= \eta \, L \, v + v \, \left( \Delta \eta - \frac{1}{2} \langle x , \nabla \eta \rangle \right) + 2 \langle \nabla \eta , \nabla v
	\rangle \notag \\
	& = \frac{1}{2} \, \eta \, v  + v \, \left( \Delta \eta - \frac{1}{2} \langle x , \nabla \eta \rangle \right) + 2 \langle \nabla \eta , \nabla v
	\rangle
\, .
\end{align}
Taking $\eta$ to have compact support, we get that
\begin{align}	\label{e:notst}
	- \int \eta \, v \, L (\eta \, v) \, \e^{\frac{-|x|^2}{4}} &=
	  - \int \left[   \frac{1}{2} \,   \eta^2 \, v^2  +
	\eta \,   v^2  \, \left( \Delta \eta - \frac{1}{2} \langle x , \nabla \eta \rangle \right) +
	\frac{1}{2}  \,  \langle \nabla \eta^2 , \nabla v^2
	\rangle  \right] \, \, \e^{\frac{-|x|^2}{4}}  \notag \\
	& =  - \int \left[    \frac{1}{2} \,   \eta^2 \, v^2  -  v^2  \, | \nabla \eta|^2 \right]
	  \, \e^{\frac{-|x|^2}{4}}
	  \, ,
\end{align}
where the second  equality uses Stokes' theorem to $ \frac{1}{2} \, \dv \left( v^2 \, \nabla \eta^2 \, \e^{\frac{-|x|^2}{4}} \right)$ to get
\begin{equation}
	 \int    \frac{1}{2} \,
	 \langle \nabla \eta^2 , \nabla v^2
	\rangle   \, \e^{\frac{-|x|^2}{4}}  = - \int v^2 \,  \left(   \eta \, \Delta \eta +   |\nabla \eta|^2 - \frac{1}{2}  \, \eta \langle x ,  \nabla \eta
	\rangle \right) \,  \e^{\frac{-|x|^2}{4}}  \, .
\end{equation}
If   $\eta$ is identically one on $B_R$ and cuts off linearly to zero on $B_{R+1} \setminus B_R$, then \eqr{e:notst} gives
\begin{equation}	\label{e:contr}
	- \int \eta \, v \, L (\eta \, v) \, \e^{\frac{-|x|^2}{4}} \leq  \int_{\Sigma \setminus B_R} v^2 \, \e^{\frac{-|x|^2}{4}}
	- \frac{1}{2} \, \int_{B_R \cap \Sigma}  v^2 \, \e^{\frac{-|x|^2}{4}} \, .
\end{equation}
However, since $|v| \leq 1$ and $\Sigma$ has polynomial volume growth, we know that
\begin{equation}
 \lim_{R \to \infty} \, \, \int_{\Sigma \setminus B_R} v^2 \, \e^{\frac{-|x|^2}{4}} = 0 \, ,
\end{equation}
so the  right-hand side of \eqr{e:contr} must be negative for all sufficiently large $R$'s.  In particular, when $R$ is large, the function $u = \eta \, v$ does not satisfy \eqr{e:Lstable}.
  This completes the proof.
\end{proof}

\end{document}